\documentclass[graybox]{svmult}

\usepackage{mathptmx}       
\usepackage{helvet}         
\usepackage{courier}        
\usepackage{type1cm}        
\usepackage{makeidx}         
\usepackage{graphicx}        
\usepackage{multicol}        
\usepackage[bottom]{footmisc}
\usepackage{enumerate}
\usepackage{amsmath,amssymb,amsfonts}%
\usepackage[all]{xy}

\usepackage{fancyhdr}
\usepackage{hyperref}
\usepackage{multirow}
\usepackage[hyperref]{ntheorem}
\usepackage{eucal}
\usepackage[numbers]{natbib}
\theoremstyle{nonumberplain}
\theorembodyfont{\upshape}
\theoremseparator{:}
\theoremstyle{plain}
\theoremheaderfont{\normalfont\bfseries}
\theorembodyfont{\upshape}
\theoremseparator{:}
\theoremclass{remark}
\theoremstyle{break}

\theoremstyle{plain} \theoremheaderfont{\normalfont\bfseries}
\theorembodyfont{\slshape} \theoremseparator{.}
\newtheorem{theo}{Theorem}[section]
\theoremclass{theo}
\theoremstyle{break}

\theoremstyle{plain} \theoremheaderfont{\normalfont\bfseries}
\theorembodyfont{\slshape} \theoremseparator{.}

\theoremstyle{plain} \theoremheaderfont{\normalfont\bfseries}
\theorembodyfont{\slshape} \theoremseparator{.}
\theoremstyle{plain} \theoremheaderfont{\normalfont\bfseries}
\theorembodyfont{\slshape} \theoremseparator{.}
\newtheorem{prop}[theo]{Proposition}

\DeclareMathOperator{\PGL}{PGL}

\DeclareMathOperator{\Aut}{Aut}

\DeclareMathOperator{\F}{\mathcal{F}}
\DeclareMathOperator{\Hom}{Hom}

\begin{document}
\title*{The cohomology of  $\mathcal{M}_{0,n}$ as an FI-module} 
\author{ Rita Jim\'enez Rolland } 
\institute{ Centro de Ciencias Matem\'aticas,\\ Universidad Nacional Aut\'onoma de M\'exico,\\ Morelia, Michoac\'an, M\'exico 58089.\\ \email{rita@matmor.unam.mx}}

\date{}

\maketitle
\begin{abstract}{ 
In this  paper we revisit the  cohomology groups of the moduli space of $n$-pointed curves of genus zero using the FI-module perspective introduced by Church-Ellenberg-Farb.  We recover known results about the corresponding representations of the symmetric group. }
\end{abstract}

\section{Introduction}\label{INTRO}

Our space of interest is $\mathcal{M}_{0,n}$, the {\it moduli space  of $n$-pointed  curves of genus zero}.  It is defined as  the quotient
$$ \mathcal{M}_{0,n}:=\F\big(\mathbb{P}^1(\mathbb{C}),n\big)/\Aut\big(\mathbb{P}^1(\mathbb{C})\big),$$
where $\F(\mathbb{P}^1(\mathbb{C}),n)$ 
 is the configuration space of $n$-ordered points in the projective line $\mathbb{P}^1(\mathbb{C})$ and the automorphism group of the projective line $\Aut\big(\mathbb{P}^1(\mathbb{C})\big)=\PGL_2(\mathbb{C})$    acts componentwise on $\F(\mathbb{P}^1(\mathbb{C}),n)$. 
For $n\geq 3$, $\mathcal{M}_{0,n}$ is a {\it fine moduli space} for 
the problem of classifying smooth $n$-pointed rational curves up to isomorphism (\cite[Proposition 1.1.2]{QUANTUM}). 


The  space $\mathcal{M}_{0,n}$  carries a natural action of the symmetric group $S_n$. The cohomology ring  of $\mathcal{M}_{0,n}$ is known and the $S_n$-representations $H^i(\mathcal{M}_{0,n};\mathbb{C})$ are  well-understood (see for example \cite{GAIFFI}, \cite{KISIN-LEHRER}, \cite{GETZLER}).

In this paper, we will consider the sequence of $S_n$-representations $H^i(\mathcal{M}_{0,n};\mathbb{C})$ as a single object, an {\it FI-module over $\mathbb{C}$}. Via this example, we introduce the basics of the FI-module theory developed by Church, Ellenberg and Farb in \cite{3AMIGOS}. We then use a well-known description of the cohomology ring of $\mathcal{M}_{0,n}$ to show in Theorem \ref{MODn} that a finite generation property  is satisfied which allows us to recover information about the $S_n$-representations in Theorem \ref{CHAR}. Specifically, we obtain a stability result concerning the decomposition  of $H^i(\mathcal{M}_{0,n};\mathbb{C})$ into irreducible $S_n$-representations, we exhibit a bound on the lengths of the representations and show that their characters have a highly constrained ``polynomial'' form.


\section{The co-FI-spaces $\mathcal{M}_{0,\bullet}$ and $\mathcal{M}_{0,\bullet+1}$}

Let {\bf FI} be the category whose objects are natural numbers {\bf n} and whose morphisms {\bf m} $\rightarrow$ {\bf n} are
injections from $[m] := \{1,\ldots,m\}$ to $[n] := \{1,\ldots,n\}$.

We are interested in the {\it co-FI-space}  $\mathcal{M}_{0,\bullet}$:  the functor from {\bf FI$^{op}$} to the category {\bf Top} of topological spaces given by given by ${\bf n} \mapsto 
\mathcal{M}_{0,n}$  that assigns to  $f:  [m] \hookrightarrow [n] $ in $\Hom_{\text{FI}} (\mathbf{m},\mathbf{n})$ the morphism $f^*: \mathcal{M}_{0,n}\rightarrow \mathcal{M}_{0,m}$ defined by $f^*\big([(p_1,p_2,\ldots, p_n)]\big)= [(p_{f(1)},p_{f(2)},\ldots, p_{f(m)})]$.
This is a particular case of the co-FI-space  $\mathcal{M}_{g,\bullet}$  considered  in \cite{JIM2} which is the functor given by ${\bf n} \mapsto 
\mathcal{M}_{g,n}$, the moduli space of Riemann surfaces of genus $g$ with $n$ marked points.

An {\it FI-module over $\mathbb{C}$} is a functor $V$ from {\bf FI} to 
 the category of $\mathbb{C}$-vector spaces $\mathbf{Vec_{\mathbb{C}}}$. Below, we denote $V(${\bf n}) by $V_n$.
Church, Ellenberg and Farb used FI-modules in \cite{3AMIGOS} to encode sequences of $S_n$-representations in single algebraic objects and with this added structure significantly strengthened the representation stability theory introduced in \cite{CHURCH_FARB}.  FI-modules translate the representation stability property  into a finite generation condition. 

 By composing the co-FI-space  $\mathcal{M}_{0,\bullet}$ with the cohomology functor $H^i(-;\mathbb{C})$, we obtain the FI-module $H^i(\mathcal{M}_{0,\bullet}) := H^i(\mathcal{M}_{0,\bullet}; \mathbb{C})$. 
We can also consider the graded version
$H^*(\mathcal{M}_{0,\bullet}):=H^*(\mathcal{M}_{0,\bullet}; \mathbb{C})$, we call this a {\it graded FI-module} over $\mathbb{C}$. 

The co-FI-space $\F(\mathbb{C},\bullet)$ given by ${\bf n} \mapsto \F(\mathbb{C},n)$, the configuration space of $n$ ordered points in $\mathbb{C}$, and the corresponding FI-modules $H^i\big(\F(\mathbb{C},\bullet)\big)$  are key in our discussion below. In the expository paper \cite{FARB}, representation stability and FI-modules  are motivated mainly through this  example. A formal discussion of FI-modules and their properties is given in \cite{3AMIGOS}. In \cite{4AMIGOS} the theory of  FI-modules is extended to modules over arbitrary Noetherian rings.  

 \medskip

\noindent {\bf The ``shifted'' co-FI-space $\mathcal{M}_{0,\bullet+1}$}. Consider the functor $\Xi_1$ from  $\mathbf{FI}$ to $\mathbf{FI}$ given by $[n]\mapsto [n]\sqcup\{0\}$. Notice that this functor induces the inclusion of groups $$J_n:S_n=\text{End}_{\text{FI}}[\mathbf{n}]\hookrightarrow \text{End}_{\text{FI}}[\mathbf{n+1}]=S_{n+1}$$ that sends the generator $(i\ i+1)$ of $S_n$ to the transposition $(i+1\ i+2)$ of $S_{n+1}$. 
In our discussion below we are interested in the ``shifted'' co-FI-space $\mathcal{M}_{0,\bullet+1}$ obtained by $\mathcal{M}_{0,\bullet}\circ\Xi_1$. Notice that  this co-FI-space is given by ${\bf n} \mapsto 
\mathcal{M}_{0,n+1}$. In the notation from \cite[Section 2]{4AMIGOS}, this means that the FI-module $$H^i(\mathcal{M}_{0,\bullet+1})=S_{+1}\big(H^i(\mathcal{M}_{0,\bullet})\big),$$  where
 $S_{+1}:\text{\bf FI-Mod}\rightarrow\text{\bf FI-Mod}$ is the shift functor given by $S_{+1}:=-\circ\Xi_1$.
The functor $S_{+1}$  performs the restriction, from an $S_{n+1}$-representation to an $S_n$-representation, consistently for all $n$ so that the resulting sequence of representations still has the structure of an FI-module. Comparing the $S_n$-representation $H^i(\mathcal{M}_{0,\bullet+1})_n$ with the $S_{n+1}$-representation $H^i(\mathcal{M}_{0,\bullet})_{n+1}=H^i(\mathcal{M}_{0,n+1})$ we have an isomorphism of $S_n$-representations

$$H^i(\mathcal{M}_{0,\bullet+1})_n\cong\text{Res}_{S_n}^{S_{n+1}} H^i(\mathcal{M}_{0,n+1}).$$

\section{Relation with the configuration space}
We will understand the cohomology ring of  $\mathcal{M}_{0,n}$  through its relation with  the configuration space $\F(\mathbb{C},n)$  of $n$ ordered points in $\mathbb{C}$.

In our descriptions below, we consider $\mathbb{P}^1(\mathbb{C})$ with coordinates $[t:z]$ and the embedding $\mathbb{C}\hookrightarrow \mathbb{P}^1(\mathbb{C})$, given by $z\mapsto [1:z]$ and let $[0:1]=\infty$. We use the brackets to indicate ``equivalence class of''.
Since there is a unique element in $\PGL_2(\mathbb{C})$ that takes any three distinct points in $\mathbb{P}^1(\mathbb{C})$ to $\big([0:1], [1:0], [1:1]\big)=(\infty, 0, 1)$,  every element in $\mathcal{M}_{0,n+1}$ can be written canonically as
$\big[\big([0:1], [1:0], [1:1], [t_1:z_1], \cdots, [t_{n-2}:z_{n-2}]\big)\big]$.  Hence, $\mathcal{M}_{0,4}\cong\mathbb{P}^1(\mathbb{C})\backslash\{\infty, 0,1\}$
and 
$\mathcal{M}_{0,n+1}\cong\F\big(\mathcal{M}_{0,4},n-2\big)$.

Define the map $\psi: \F(\mathbb{C},n)\longrightarrow \mathcal{M}_{0,n+1}$ by 
$$\psi (z_1,z_2,\ldots,z_n)=\Big[\Big(\infty, 0, 1, \frac{z_3-z_1}{z_2-z_1},\frac{z_4-z_1}{z_2-z_1},\cdots,\frac{z_n-z_1}{z_2-z_1}\Big)\Big].$$

 The symmetric group $S_n$ acts on $\F(\mathbb{C},n)$ by permuting the coordinates. Let $(1\ 2), (2\ 3),\ldots (n-1\ n)$ be transpositions generating $S_n$ and notice that
$$\begin{matrix}
&\psi \big((1\ 2)\cdot (z_1,z_2,\ldots,z_n)\big)=\Big[\Big(\infty, 0, 1, \frac{z_3-z_2}{z_1-z_2},\frac{z_4-z_2}{z_1-z_2},\cdots,\frac{z_n-z_2}{z_1-z_2}\Big)\Big]\\
&=\Bigg[\Big[\begin{matrix} 1&0\\1&-1\end{matrix}\Big]\cdot\Big([0:1], [1:0], [1:1], [1: \frac{z_3-z_2}{z_1-z_2}],[1:\frac{z_4-z_2}{z_1-z_2}],\cdots,[1:\frac{z_n-z_2}{z_1-z_2}]\Big)\Bigg]\\
&=\Bigg[\Big([0:1], [1:1], [1:0], [1: \frac{z_3-z_1}{z_2-z_1}],[1:\frac{z_4-z_1}{z_2-z_1}],\cdots,[1:\frac{z_n-z_1}{z_2-z_1}]\Big)\Bigg]\\
&=(2\ 3)\cdot \psi(z_1,z_2,\ldots,z_n)

\end{matrix}$$
and in general 
$$\psi \big((i\ i+1)\cdot (z_1,z_2,\ldots,z_n)\big)=(i+1\ i+2)\cdot \psi(z_1,z_2,\ldots,z_n)\text{\ \  for\ \ } i\geq 2.$$

Therefore the map $\psi: \F(\mathbb{C},n)\longrightarrow \mathcal{M}_{0,n+1}$ is equivariant with respect to the inclusion $J_n:S_n \hookrightarrow S_{n+1}$.  In other words, 
$\psi: \F(\mathbb{C},\bullet)\longrightarrow \mathcal{M}_{0,\bullet+1}$ is a map of co-FI-spaces.\medskip

\noindent{\bf Relation with the Coxeter arrangement of type $A_{n-1}$.}
 The {\it complement of the complexified Coxeter arrangement of hyperplanes type $A_{n-1}$} is $M(\mathcal{A}_{n-1})$,  the image of $\F(\mathbb{C},n)$ under the quotient map $\mathbb{C}^n\rightarrow\mathbb{C}^n/N$, where $N=\{(z_1,\ldots,z_n)\in\mathbb{C}^n: z_i=z_j \text{ for  }1\leq i,j\leq n\}$.

As explained in \cite{GAIFFI}, it turns out  that  the moduli space $\mathcal{M}_{0,n+1}$ is also in bijective correspondence with the projective arrangement
$$M(d\mathcal{A}_{n-1}):=\pi\big(M(\mathcal{A}_{n-1})\big)\cong M(\mathcal{A}_{n-1})/\mathbb{C}^*,$$
where $\pi: \mathbb{C}^{n-1}\backslash\{0\}\rightarrow \mathbb{P}^{n-2}(\mathbb{C})$ is the Hopf bundle projection, which takes $z\in\mathbb{C}^{n-1}\backslash\{0\}$ to $\lambda z$ for $\lambda\in\mathbb{C}^*$.
Moreover,  the map $\psi$ factors through  $M(\mathcal{A}_{n-1})$ and  $M(d\mathcal{A}_{n-1})$.
$$
\xymatrix{ \F(\mathbb{C},n)\ar[r]\ar[ddr]_{\psi} & M(\mathcal{A}_{n-1})\ar[d]^{\pi}\\ & M(d \mathcal{A}_{n-1})\ar[d]^{\cong} \\
&\mathcal{M}_{0,n+1}}
$$

In \cite{GAIFFI}, Gaiffi  extends the natural $S_n$-action on $H^*\big(M(\mathcal{A}_{n-1});\mathbb{C}\big)$ to an $S_{n+1}$-action using  the vertical map in the diagram above and the natural $S_{n+1}$-action on $H^*\big(\mathcal{M}_{0,n+1};\mathbb{C}\big)$.
\medskip


\noindent{\bf The cohomology rings.}
As proved in \cite[Prop. 2.2 \& Theorem 3.2]{GAIFFI} , the map $\psi$ allows us to relate the cohomology rings of $\mathcal{M}_{0,n+1}$ and $\F(\mathbb{C};n)$.  See also \cite[Cor. 3.1] {GETZLER}.

\begin{prop}\label{ISO} The maps $\psi$ induces an isomorphism of cohomology rings
$$H^*\big(\F(\mathbb{C},n);\mathbb{C}\big)\cong H^*\big(\mathcal{M}_{0,n+1};\mathbb{C}\big)\otimes H^*(\mathbb{C}^*;\mathbb{C})$$
as $S_n$-modules. The symmetric group $S_n$ acts trivially on $H^*(\mathbb{C}^*;\mathbb{C})$ 
and acts on $H^*\big(\mathcal{M}_{0,n+1};\mathbb{C}\big)$ through the inclusion $J_n:S_n \hookrightarrow S_{n+1}$ that sends the generator $(i\ i+1)$ of $S_n$ to the transposition $(i+1\ i+2)$ of $S_{n+1}$.
\end{prop}

This means that the map of co-FI-spaces 
$\psi: \F(\mathbb{C},\bullet)\longrightarrow \mathcal{M}_{0,\bullet+1}$
induces an isomorphism of graded FI-modules
$$H^*\big(\F(\mathbb{C},\bullet)\big)\cong H^*\big(\mathcal{M}_{0,\bullet+1}\big)\otimes  H^*(\mathbb{C}^*),$$
where  $H^*(\mathbb{C}^*)$ 
is the trivial graded FI-module given by $\mathbf{n}\mapsto H^*(\mathbb{C}^*;\mathbb{C})$.

Furthermore, Arnol'd obtained a presentation of the  cohomology ring of $\F(\mathbb{C},n)$ in  (\cite{ARNOLD}).
 
\begin{theo}\label{ARNPRES} The cohomology ring $ H^*\big(\F(\mathbb{C},n);\mathbb{C}\big)$ is isomorphic to  the $\mathbb{C}$-algebra $\mathcal{R}_n$ generated by $1$ and forms $\omega_{i,j} :=\frac{d\log(z_j-z_i)}{2\pi \text{i}}$ , $1\leq i\neq j\leq n$, with relations $\omega_{i,j}=\omega_{j,i}$,\ \ $\omega_{i,j}\omega_{k,l}=-\omega_{k,l}\omega_{i,j}$ and $\omega_{i,j}\omega_{j,k}+\omega_{j,k}\omega_{k,i}+\omega_{k,i}\omega_{i,j}=0$. The action $S_n$ is given by $\sigma\cdot\omega_{i,j}=\omega_{\sigma(i)\sigma(j)}$ for $\sigma\in S_n$.

\end{theo}

As a consequence of the isomorphism in Proposition \ref{ISO}, we also have a concrete description of the cohomology ring $H^*\big(\mathcal{M}_{0,n+1};\mathbb{C}\big)$.  We refer the reader to \cite[Cor. 3.1] {GETZLER}, \cite[Theorem 3.4]{GAIFFI} and references therein.

\begin{theo}\label{MODULIPRES} The cohomology ring $H^*\big(\mathcal{M}_{0,n+1};\mathbb{C}\big)$
is isomorphic to the subalgebra of $\mathcal{R}_n$ generated by $1$ and elements $\theta_{i,j}:=\omega_{i,j} -\omega_{12}$  for $\{i,j\}\neq\{1,2\}$. The $S_n$-action given by $\sigma\cdot\theta_{i,j}=\theta_{\sigma(i),\sigma(j)} -\theta_{\sigma(1),\sigma(2)}$, for $\sigma\in S_n$.\end{theo}

\section{Finite generation}

We can  use the explicit presentations in Theorems \ref{ARNPRES} and \ref{MODULIPRES} to understand the FI-modules $H^i\big(\F(\mathbb{C},\bullet)\big)$ and $H^i(\mathcal{M}_{0,\bullet+1})$.\medskip

An FI-module $V$ over $\mathbb{C}$  is said to be {\it finitely generated in degree $\leq m$} if there exist $v_1,\ldots,v_s$,  with each $v_i\in V_{n_i}$ and $n_i\leq m$, such that $V$ is the minimal sub-FI-module of $V$ containing $v_1,\ldots,v_s$.  Finitely generated FI-modules have strong closure properties:  
extensions and quotients of finitely generated FI-modules are still finitely generated and finite generation is preserved when taking sub-FI-modules.

Notice that from Theorem \ref{ARNPRES} it follows that $H^1\big(\F(\mathbb{C},n);\mathbb{C}\big)$ is generated as an $S_n$-module by the class $\omega_{1,2}$. Therefore,  the FI-module $H^1\big(\F(\mathbb{C},\bullet)\big)$ is finitely generated in degree $2$ by the class $\omega_{1,2}$ in $H^1\big(\F(2;\mathbb{C})\big)$.  

Similarly, from Theorem \ref{MODULIPRES} we know that 
$H^1(\mathcal{M}_{0,n+1})$ is generated by the $\theta_{i,j}$ classes and notice that 
$\theta_{1,j}=(j\ 3)\cdot\theta_{1,3}$ for $j\neq\{1,2\}$; $\theta_{2,j}=(j\ 3)\cdot\theta_{2,3}$ for $j\neq\{1,2\}$ and $\theta_{i,j}=(i\ 3)(j\ 4)\cdot\theta_{3,4}$ for $\{i,j\}\neq\{1,2\}$. 
Therefore,  $H^1(\mathcal{M}_{0,n+1})$ is generated by  $\theta_{1,3}$, $\theta_{2,3}$ and $\theta_{3,4}$ as an $S_n$-module. 
This means that the FI-module $H^1(\mathcal{M}_{0,\bullet+1})$ is finitely generated by the classes $\theta_{1,3}$, $\theta_{2,3}$ and $\theta_{3,4}$ in $H^1(\mathcal{M}_{0,\bullet+1})_4$, hence in degree $4$. \medskip

An FI-module $V$ encodes the information of the sequence $V_n$ of $S_n$-representations. Finite generation of $V$ puts strong constraints on the decomposition of each $V_n$ into irreducible representations  and its character.\medskip

\noindent {\bf Notation for representations of $S_n$. } The irreducible representations of $S_n$ over  $\mathbb{C}$ are classified by partitions of $n$. A partition $\lambda$ of $n$ is a set of positive integers $\lambda_1\geq\cdots\geq\lambda_l >0$ where $l\in\mathbb{Z}$ and $\lambda_1+\cdots+\lambda_l=n$.  We  write $|\lambda|=n$. The corresponding irreducible $S_n$-representation will be denoted by $V_{\lambda}$. Every  $V_{\lambda}$ is defined over $\mathbb{C}$ and any $S_n$-representation decomposes over $\mathbb{C}$ into a direct sum of irreducibles.

If $\lambda$ is any  partition of $m$, i.e. $|\lambda|=m$, then for any $n\geq |\lambda|+ \lambda_1$ the \textit{padded partition} $\lambda[n]$ of $n$ is given by $n-|\lambda|\geq\lambda_1\geq\cdots\geq\lambda_l>0$. Keeping the notation from \cite{CHURCH_FARB},  we set $V(\lambda)_n=V_{\lambda[n]}$  for any $n\geq |\lambda|+\lambda_1$. Every irreducible $S_n$-representation is of the form $V(\lambda)_n$ for a unique partition $\lambda$.  
We define the {\it length} of an irreducible representation of $S_n$ to be the
number of parts in the corresponding partition of $n$. The trivial representation  has length $1$, and the alternating representation has length $n$. We define the {\it length $\ell(V )$} of a finite dimensional representation $V$ of $S_n$ to be the maximum of the lengths of the irreducible constituents. 
\bigskip

We say that an FI-module  $V$ over $\mathbb{C}$ has {\it weight $\leq d$}  if for every $n\geq 0$ and every irreducible constituent $V(\lambda)_n$ we have $|\lambda|\leq d$. The degree of generation of an FI-module $V$ gives an upper bound for the weight (\cite[Prop. 3.2.5]{3AMIGOS}). The weight of an FI-module is closed under subquotients and extensions. Moreover, if a finitely generated FI-module $V$ has weight $\leq d$, by definition, $\ell (V_n)\leq d+1$ for all $n$ and the alternating representation cannot not appear in the decomposition into irreducibles of $V_n$ once $n>d+1$.

Notice that the FI-module $H^1\big(\F(\mathbb{C}, \bullet)\big)$ has weight at most $2$  and so does $H^1(\mathcal{M}_{0,\bullet+1})$, since it is a sub-FI-module of $H^1\big(\F(\mathbb{C},\bullet)\big)$.\medskip
 
An FI-module $V$ has {\it stability degree} $\leq N$, if for every $a\geq 0$ and $n\geq N+a$ the map of coinvariants
\begin{equation}\label{MAPI}
(I_n)_*:(V_n)_{S_{n-a}}\rightarrow (V_{n+1})_{S_{(n+1)-a}}
\end{equation}
induced by the standard inclusion $I_n:\{1,\ldots n\}\hookrightarrow\{1,\ldots, n, n+1\}$, is an isomorphism of $S_a$-modules (see \cite[Definition 3.1.3]{3AMIGOS} for a more general definition). Here,  
$S_{n-a}$ is the subgroup of $S_n$ that permutes $\{a+1,\ldots,n\}$ and acts trivially on $\{1,2,\ldots,a\}$. The coinvariant quotient $(V_n)_{S_{n-a}}$ is the $S_a$-module $V_n\otimes_{\mathbb{C}[S_{n-a}]}\mathbb{C}$,  the largest quotient of $V_n$ on which $S_{n-a}$ acts trivially.\medskip


The finite generation properties of the FI-modules $H^i\big(\F(\mathbb{C},\bullet)\big)$ have already been  discussed  in \cite[Example 5.1.A]{3AMIGOS}.\medskip

\begin{prop}\label{CONF}
The FI-module  $H^i\big(\F(\mathbb{C}, \bullet)\big)$ is finitely generated with weight $\leq 2i$ and has stability degree $\leq 2i$ 
\end{prop}
\begin{proof}
From Theorem \ref{ARNPRES} the graded FI-module $H^*\big(\F(\mathbb{C},\bullet)\big)$ is generated by the FI-module $H^1\big(\F(\mathbb{C},\bullet)\big)$ that has weight $\leq 2$. It follows by \cite[Theorem 4.2.3]{3AMIGOS} that $H^i\big(\F(\mathbb{C},\bullet)\big)$ is finitely generated with weight $\leq 2i$. Moreover, in \cite{3AMIGOS} it is shown that $H^i\big(\F(\mathbb{C},\bullet)\big)$ has the additional structure of what \cite{3AMIGOS} calls an FI$\#$-module, which implies that it has stability degree bounded above by the weight (see proof of \cite[Cor. 4.1.8]{3AMIGOS}).
\end{proof}\medskip

Finite generation for the FI-modules $H^i(\mathcal{M}_{0,\bullet+1})$ follows from Theorem \ref{MODULIPRES} and Proposition \ref{CONF}.

\begin{theo}\label{MODShift}
The FI-module $H^i(\mathcal{M}_{0,\bullet+1})$ is finitely generated  in degree $\leq 4i$, with weight $\leq 2i$ and has stability degree $\leq 2i$.
\end{theo}
\begin{proof}

By Theorem \ref{MODULIPRES} the graded FI-module $H^*(\mathcal{M}_{0,\bullet+1})$ is  generated by the FI-module $H^1(\mathcal{M}_{0,\bullet+1})$, which is finitely generated in degree $\leq 4$ and has weight $\leq 2$. It follows from  \cite[Proposition 2.3.6]{3AMIGOS} that the FI-module $H^i(\mathcal{M}_{0,\bullet+1})$ is finitely generated in degree $\leq 4i$. By \cite[Corolary 4.2.A]{3AMIGOS}  it has weight $\leq 2i$.
  
 Moreover,  from \cite[Lemma 3.1.6]{3AMIGOS} we have that the stability degree of $H^i(\mathcal{M}_{0,\bullet+1})$ is bounded above by the stability degree of $H^i\big(\F(\mathbb{C},\bullet) \big)$.
 \end{proof}\medskip

\noindent{\bf From $H^i(\mathcal{M}_{0,\bullet+1})$ to the FI-module $H^i(\mathcal{M}_{0,\bullet})$}. 
The relation between the degree of generation of an  FI-module $V$ and its ``shift'' $S_{+1}V$ was established in \cite[Cor. 2.13]{4AMIGOS}.  We can also relate the weights and stability degrees using the classical branching rule (see e.g. \cite{FULTON_HARRIS}). 
\begin{prop}\label{BRANCH}
Let $\lambda$ be a partition of $n+1$ and $V_\lambda$ the corresponding irreducible $S_{n+1}$-representation, then as $S_{n}$-representations we have the decomposition
$$\text{Res}_{S_{n}}^{S_n+1}V_\lambda\cong\bigoplus_\nu V_\nu$$
over those partitions $\nu$ of $n$ obtained from $\lambda$ by removing one box from one of the columns of the corresponding Young diagram. 
\end{prop} 

\begin{theo}[Finite generation and ``shifted'' FI-modules] \label{SHIFT}Let $V$ be a finitely generated  FI-module generated in degree $\leq d$, then $S_{+1}V$ is finitely generated in degree $\leq d$. Conversely, if the FI-module $S_{+1}V$ is finitely generated in degree $\leq d$, then  $V$ is finitely generated in degree $\leq d+1$.

Furthermore, if $S_{+1}V$ has weight $\leq M$ and stability degree $\leq N$, then $V$  has weight $\leq M+1$ and stability degree $\leq N+1$. Conversely, if $V$ has weight $\leq M$ and stability degree $\leq N$, then $S_{+1}V$  has weight $\leq M$ and stability degree $\leq N$.
\end{theo}
\begin{proof}
If $V$ has weight $\leq M$, then for all $n\geq 0$, the irreducible components $V(\mu)_{n+1}$ of $V_{n+1}$ have $|\mu|\leq M$. From Proposition \ref{BRANCH} it follows that $\text{Res}_{S_{n}}^{S_n+1}V(\mu)_{n+1}$ will be a direct sum of irreducibles $V(\lambda)_n$, with $|\lambda|\leq|\mu|\leq M$. Conversely, if $S_{+1}V$  has weight $\leq M$, then each irreducible component $V(\lambda)_n$ of $S_{+1}V_n$ has $|\lambda|\leq M$. By Proposition \ref{BRANCH}, it comes from the restriction of some $V(\mu)_{n+1}$ with $|\mu|\leq|\lambda|+1\leq M+1$.


On the other hand, the functor $\Xi_1$ sends $\{1,\ldots,a\}$ into $\{2,\dots,a+1\}$ and  $\{a+1,\ldots,n\}$ into $\{a+2,\dots,n+1\}$. Therefore, the inclusion $J_n:S_n\hookrightarrow S_{n+1}$ maps the subgroup $S_{n-a}$ of $S_n$ onto the subgroup $S_{(n+1)-(a+1)}$ of $S_{n+1}$ and we have that 
$$(V_{n+1})_{S_{(n+1)-(a+1)}}=V_{n+1} \otimes_{\mathbb{C}[S_{(n+1)-(a+1)}]}\mathbb{C}= S_{+1}(V)_n\otimes_{\mathbb{C}[S_{n-a}]}\mathbb{C}=\big(S_{+1}(V)_n\big)_{S_{n-a}},$$
which implies the statement about stability degrees.
\end{proof}\medskip

In \cite{JIM2} we proved finite generation for the FI-modules $H^i(\mathcal{M}_{g,\bullet})$ when $g\geq 2$.  The case when $g=0$ follows from Theorem \ref{SHIFT} and Theorem \ref{MODShift}.

\begin{theo}\label{MODn}
The FI-module $H^i(\mathcal{M}_{0,\bullet})$ is finitely generated with weight $\leq 2i+1$
and has stability degree $\leq 4i$.
\end{theo}

\noindent{\bf The first cohomology group.} 
Recall that $H^1(\mathcal{M}_{0,\bullet+1})_n$ is generated by the classes $\theta_{i,j}=\omega_{i,j}-\omega_{1,2}$ and it is a subrepresentation of $H^1(\F(\mathbb{C},n))$ which has a basis given by the classes $\omega_{i,j}$.
In particular, notice that $\dim H^1(\mathcal{M}_{0,\bullet+1})_n=\dim H^1(\F(\mathbb{C},n))-1$. Moreover for $n\geq 4$, we have the decomposition
$$H^1(\F(\mathbb{C},n))=V(0)_n\oplus V(1)_n\oplus V(2)_n.$$
Then,  for $n\geq 4$ the $S_n$-representation $$H^1(\mathcal{M}_{0,\bullet+1})_n=V(1)_n\oplus V(2)_n\cong\text{Res}_{S_n}^{S_{n+1}} H^1(\mathcal{M}_{0,n+1}).$$ 
Proposition \ref{BRANCH} implies that   for $n\geq 4$, we have  that $H^1(\mathcal{M}_{0,n+1})=V(2)_{n+1}$ as a representation of $S_{n+1}$. 
Moreover, notice that  $H^1(\mathcal{M}_{0,n+1})$ is finitely generated by the classes $\theta_{1,3}$, $\theta_{2,3}$ and $\theta_{3,4}$ in $H^1(\mathcal{M}_{0,5})$ not only an an $S_n$-module, but also as an $S_{n+1}$-module. Therefore, the FI-module $H^1(\mathcal{M}_{0,\bullet})$ is finitely generated in degree$\leq 5$ and has weight $\leq 2$.

\section{The $S_n$-representations $H^i(\mathcal{M}_{0,n};\mathbb{C})$}

At this point we can apply the theory of FI-modules to the  finitely generated FI-modules $H^i(\mathcal{M}_{0,\bullet})$ and $H^i(\mathcal{M}_{0,\bullet+1})$ to obtain information about the corresponding sequences of $S_n$-representations and their characters. The following result is a direct consequence from \cite[Prop. 3.3.3 and Theorem 3.3.4]{3AMIGOS} and Theorems \ref{MODShift} and \ref{MODn}.

\begin{theo}\label{CHAR} Let $i\geq 0$. For  $n\geq 4i+2$,  the sequence $\big\{H^i(\mathcal{M}_{0,n}) \big\}$ of representations of $S_n$ and the sequence  $\big\{H^i(\mathcal{M}_{0,\bullet+1})_{n-1} \big\}$ of $S_{n-1}$-representations satisfy the following:

{\bf (a)} The decomposition into irreducibles of $H^i(\mathcal{M}_{0,n};\mathbb{C})$  and of $H^i(\mathcal{M}_{0,\bullet+1};\mathbb{C})_{n-1}$ stabilize in the sense of uniform representation stability (\cite{CHURCH_FARB}) with stable range $n\geq 4i+2$.


{\bf (b)} The length of $H^i(\mathcal{M}_{0,\bullet+1};\mathbb{C})_{n-1}$ is bounded above by $2i$ and the length of   $ H^i(\mathcal{M}_{0,n};\mathbb{C})$ is bounded above by $2i+1$.

{\bf (c)} The sequence of characters  of the representations $H^i(\mathcal{M}_{0,\bullet+1};\mathbb{C})_{n-1}$ and $H^i(\mathcal{M}_{0,n};\mathbb{C})$ are {\it eventually polynomial}, in the sense that there exist {\it character polynomials} $P_i(X_1,X_2,\ldots,X_r)$ and $Q_i(X_1,X_2,\ldots,X_s)$ in the cycle-counting functions $X_k(\sigma):=$(number of $k$-cycles in $\sigma$) such that for all $n\geq 4i+2$:

$$\chi_{H^i(\mathcal{M}_{0,\bullet+1};\mathbb{C})_{n-1}}(\sigma)=P_i(X_1,X_2,\ldots,X_r)(\sigma) \text{ \ \ for all \ \ }\sigma\in S_{n-1},\text{\ \ \ and} $$

$$\chi_{H^i(\mathcal{M}_{0,n};\mathbb{C})}(\sigma)=Q_i(X_1,X_2,\ldots,X_s)(\sigma) \text{ \ \ for all \ \ }\sigma\in S_n .$$  Moreover, the degree of $P_i$ is $\leq 2i$ and the degree of $Q_i$ is $\leq 2i+1$, where we take $\deg X_k=k$. In particular, $r\leq 2i$ and $s\leq 2i+1$. 
\medskip 

\end{theo}

If $e\in S_{n-1}$ is the identity element,  from Theorem \ref{CHAR}(c), we obtain that the dimensions 
$$\dim_{\mathbb{C}}\big(H^i(\mathcal{M}_{0,n};\mathbb{C})\big)=
\chi_{H^i(\mathcal{M}_{0,\bullet+1};\mathbb{C})_{n-1}}(e)=P_i(X_1(e),\ldots,X_r(e))=P_i(n-1,\ldots,0)$$
are polynomials in $n$ of degree $\leq 2i$. This agrees with the known  Poincar\'e polynomial of $\mathcal{M}_{0,n}$ (see\cite[Cor. 2.10]{KISIN-LEHRER} and also \cite[5.5(8)]{GETZLER}).

From Theorem \ref{MODn} and the definition of weight,  we recover the fact that the alternating representation does not appear in the cohomology of $\mathcal{M}_{0,n}$ (\cite[Prop. 2.16]{KISIN-LEHRER}).

Theorem \ref{CHAR}(a) implies that the dimensions of the vector spaces $H^i(\mathcal{M}_{0,n}/S_n;\mathbb{C})$ and $H^i(\mathcal{M}_{0,n+1}/S_n;\mathbb{C})$ are constant. For the sequence $\{\mathcal{M}_{0,n}/S_n\}$, this is actually a trivial consequence from the fact that $\mathcal{M}_{0,n}/S_n$ has the cohomology of a point as shown in \cite[Theorem 2.3]{KISIN-LEHRER}.\medskip

\noindent{\bf Recursive relation for characters}.
In \cite[Theorem 4.1]{GAIFFI}, Gaiffi obtained a recursive formula that connects the characters of the $S_n$-representations $H^*(\mathcal{M}_{0,\bullet+1})_n$
and $H^*(\mathcal{M}_{0,n})$ as follows
\begin{equation}\label{REC}
\chi_{H^i(\mathcal{M}_{0,\bullet+1})_n}=\chi_{H^i(\mathcal{M}_{0,n})}+\big(X_1-1\big)\cdot\chi_{H^{i-1}(\mathcal{M}_{0,n})}\text{\ \ \ \ for }n\geq 3.
\end{equation}

In particular,  we know that 
$\chi_{H^1(\mathcal{M}_{0,\bullet+1})_n}=\chi_{H^1(\F(\mathbb{C},n))}-1={X_1 \choose 2}+X_2-1$ when $n\geq 4$. Therefore, for $i=1$, the recursive formula (\ref{REC}) gives us the character polynomial of degree $2$

$$\chi_{H^1(\mathcal{M}_{0,n})}=\chi_{H^1(\mathcal{M}_{0,\bullet+1})_n} - \big(X_1-1\big)\cdot\chi_{H^{0}(\mathcal{M}_{0,n})}
={X_1 \choose 2}+X_2-X_1=\chi_{V(2)}$$

as expected since  $H^1(\mathcal{M}_{0,n})=V(2)_{n}$.

Furthermore, if $P_i$ and $Q_i$ are the character polynomials of $H^i(\mathcal{M}_{0,\bullet+1})_n$
and $H^i(\mathcal{M}_{0,n})$ from Theorem \ref{CHAR} (c) for $n\geq 4i+2$, then formula (\ref{REC}) can be written as
$Q_i=P_i-(X_1-1)\cdot Q_{i-1}$ and $\deg Q_i\leq max\big(\deg P_i, 1+\deg Q_{i-1}\big)\leq 2i $. As a consequence of this and Theorem \ref{CHAR} (c) we have that,  for $n\geq 4i+2$, the values of $\chi_{H^i(\mathcal{M}_{0,\bullet+1};\mathbb{C})_{n}}(\sigma)$ and  $\chi_{H^i(\mathcal{M}_{0,n};\mathbb{C})}(\sigma)$ depend only on ``short cycles'', i.e. cycles on $\sigma$ of  length $\leq 2i$.
\medskip

\noindent{\bf More is known about the $S_n$-representations.} In this paper we were mainly interested in highlighting the methods, since 
more precise information about the characters of the $S_n$-representations is known. 
The  moduli space $\mathcal{M}_{0,n}$ can be represented by a finite type $\mathbb{Z}$-scheme and the manifold $\mathcal{M}_{0,n}(\mathbb{C})$  of $\mathbb{C}$-points of this scheme  corresponds to the definition in Section \ref{INTRO}. In \cite{KISIN-LEHRER} Kisin and Lehrer used an equivariant comparison theorem in $\ell$-adic cohomology  and the Grothendieck-Lefschetz's fixed point formula to obtain explicit descriptions of the graded character of the $S_n$-action on the cohomology of $\mathcal{M}_{0,n}(\mathbb{C})$ via counts of number of points of varieties over finite fields.
With their techniques they obtain the Poincar\'e polynomial of a permutation in $S_n$ of a specific cycle type acting on $H^*(\mathcal{M}_{0,n};\mathbb{C})$ (\cite[Theorem 2.9]{KISIN-LEHRER}) and a description of the top cohomology $H^{n-3}(\mathcal{M}_{0,n};\mathbb{C})$ \cite[Proposition 2.18]{KISIN-LEHRER}.
 Furthermore, Getzler uses the language of operads in \cite{GETZLER}  to obtain formulas for the characters of the $S_n$-modules $H^i(\mathcal{M}_{0,n};\mathbb{C})$.\medskip


\noindent{\bf The cohomology of $\overline{\mathcal{M}}_{0,n}$.}
A related space of interest is $\overline{\mathcal{M}}_{0,n}$, the {\it Deligne-Mumford compactification} of $\mathcal{M}_{0,n}$. It is a {\it fine moduli space} for 
stable $n$-pointed rational curves for $n\geq 3$ (see \cite[Chapter 1]{QUANTUM} and reference therein). 
It can also be constructed from $M(d\mathcal{A}_{n-1})$ using the theory of wonderful models of hyperplanes arrangements developed by De Concini and Procesi (see for example  \cite[Chapter 2]{GAIFFI3}).  The space $\overline{\mathcal{M}}_{0,n}$  also carries a natural action of the symmetric group $S_n$. Hence, a natural question to ask is whether the FI-module theory could tell us something about its cohomology groups as  $S_n$-representations.

 Explicit  presentations of the cohomology ring of the manifold of complex points $\overline{\mathcal{M}}_{0,n}(\mathbb{C})$  have been obtained by Keel  \cite{KEEL} and Yuzvinsky \cite{YUV}. Moreover,  several recursive and generating formulas for the  Poincar\'e polynomials have been computed  (for instance see \cite{YUV},\cite{GETZLER}, \cite{MANIN}, \cite{GAIFFICALL}). 
The sequence  $H^i(\overline{\mathcal{M}}_{0,n}(\mathbb{C});\mathbb{C})$ has the structure of an FI-module, however,  the Betti numbers of $\overline{\mathcal{M}}_{0,n}(\mathbb{C})$ grow exponentially in $n$,   
which precludes finite generation. Therefore an analogue of Theorem \ref{CHAR} cannot be obtained for this space.

On the other hand, as observed  in \cite{ETINGOF-ETAL}, the manifold $\overline{\mathcal{M}}_{0,n}(\mathbb{R})$ of real points of $\overline{\mathcal{M}}_{0,n}$ is topologically  similar to $\F(\mathbb{C}, n-1)$, the configuration space   of $n-1$ ordered points in $\mathbb{C}$, in the sense that   both are $K(\pi,1)$-spaces, have Poincar\'e polynomials with a simple factorization and Betti numbers that grow polynomially in $n$.  The cohomology ring of the real locus $\overline{\mathcal{M}}_{0,n}(\mathbb{R})$ was  completely determined in \cite{ETINGOF-ETAL} and an explicit formula for the graded character of the $S_n$-action was obtained in  \cite{RAINS}. The presentation  of the cohomology ring given in \cite{ETINGOF-ETAL} can be used to prove finite generation for the FI-modules $H^i(\overline{\mathcal{M}}_{0,n}(\mathbb{R});\mathbb{C})$ and to obtain an analogue of Theorem \ref{CHAR} for this space (see \cite{JIMAYA}).\bigskip

\begin{acknowledgement}
I would like to thank to Alex Suciu for pointing out a  relevant reference and Jennifer Wilson for useful comments. I am grateful to the Department of Mathematics at Northeastern University for providing such appropriate working conditions that allowed the completion of this paper.
\end{acknowledgement}

\bibliographystyle{plain}
\bibliography{referFI}

\end{document}